\documentclass[a4paper,12pt]{article}
\usepackage{mathtools}
\usepackage{amsfonts}
\usepackage{amsthm}
\usepackage{graphicx}
\usepackage{framed}
\usepackage{float}
\usepackage{upgreek}
\usepackage{enumerate}
\usepackage{MnSymbol}
\usepackage{amsbsy}
\usepackage{fullpage}
\usepackage{bbm}
\usepackage{url}
\theoremstyle{plain}
\newtheorem{thm}{Theorem}
\newtheorem{lemma}[thm]{Lemma}

\newtheorem*{l4'}{Lemma E4'}

\theoremstyle{definition}
\newtheorem{defi}[thm]{Definition}
\newtheorem{rmk}[thm]{Remark}

\theoremstyle{remark}

\begin{document}

\title{A note on the ``conditional triviality'' property for regular conditional distributions}
\author{Julian Newman}

\maketitle

\begin{abstract}
\noindent We make the simple, and yet deep, observation that a regular conditional distribution (rcd) almost surely trivialises the conditioning $\sigma$-algebra if and only if there exists a ``measurable selection'' of regular conditional distributions for almost all the measures arising out of the original rcd.
\end{abstract}

\noindent Let $(X,\Sigma)$ be a countably generated measurable space, and let $\mathcal{P}$ be the set of probability measures on $(X,\Sigma)$, equipped with its natural $\sigma$-algebra $\mathfrak{K}:=\sigma(\mu \mapsto \mu(A) : A \in \Sigma)$. Given a probability measure $\rho \in \mathcal{P}$ and a sub-$\sigma$-algebra $\mathcal{G}$ of $\Sigma$, we will say that $\mathcal{G}$ is \emph{$\rho$-trivial} if $\rho(G) \in \{0,1\}$ for all $G \in \mathcal{G}$.
\\ \\
From now on, let $\rho$ be a probability measure on $(X,\Sigma)$, and let $\mathcal{G}$ be a sub-$\sigma$-algebra of $\Sigma$. (The restriction of $\rho$ to $\mathcal{G}$, which is a probability measure on $(X,\mathcal{G})$, is denoted $\rho|_{\mathcal{G}}$.)

\begin{defi}
A \emph{regular conditional distribution of $\rho$ given $\mathcal{G}$} is a $(\mathcal{G},\mathfrak{K})$-measurable function $\rho^\mathcal{G}:X \to \mathcal{P}$ such that for all $A \in \Sigma$ and $G \in \mathcal{G}$,
\[ \rho(A \cap G) \ = \ \int_G \rho^\mathcal{G}(x)(A) \, \rho(dx). \]
\end{defi}

\noindent Since $(X,\Sigma)$ is countably generated, it is easy to show (using the $\pi$-$\lambda$ theorem) that if such a function $\rho^\mathcal{G}:X \to \mathcal{P}$ exists then it is $\rho|_{\mathcal{G}}$-essentially unique. It is also well-known that if $(X,\Sigma)$ is standard, then such a function $\rho^\mathcal{G}$ \emph{must} exist (see e.g.~Theorem~33.3 of [1]).
\\ \\
From now on, $\rho^\mathcal{G}$ denotes a rcd of $\rho$ given $\mathcal{G}$.

\begin{rmk}
Given any set $S$ and a function $g:X \to S$, we will say that $g$ is \emph{$\rho$-almost constant} if there exists $c \in S$ such that $g(x)=c$ for $\rho$-almost all $x \in X$. It is clear (but still worth pointing out explicitly) that the following are equivalent: (i)~$\mathcal{G}$ is $\rho$-trivial; (ii)~the constant map $x \mapsto \rho$ is a rcd of $\rho$ given $\mathcal{G}$; (iii)~$\rho^\mathcal{G}:X \to \mathcal{P}$ is $\rho$-almost constant.
\end{rmk}

\begin{lemma}
For any $E \in \mathcal{G} \otimes \Sigma$,
\[ \rho( x \in X : (x,x) \in E ) \ = \ \int_X \int_X \mathbbm{1}_E(x,y) \, \rho^\mathcal{G}(x)(dy) \, \rho(dx). \]
\end{lemma}

\begin{proof}
For any $G \in \mathcal{G}$ and $A \in \Sigma$,
\begin{align*}
\rho( x \in X : (x,x) \in G \times A ) \ &= \ \rho(A \cap G) \\
&= \ \int_G \rho^\mathcal{G}(x)(A) \, \rho(dx) \\
&= \ \int_X \int_X \mathbbm{1}_{G \times A}(x,y) \, \rho^\mathcal{G}(x)(dy) \, \rho(dx).
\end{align*}
\noindent So then, the collection $\mathcal{E}$ of sets $E \in \mathcal{G} \otimes \Sigma$ satisfying the desired result includes all sets of the form $G \times A$ (with $G \in \mathcal{G}$ and $A \in \Sigma$). But also, by the monotone convergence theorem, $\mathcal{E}$ is a $\lambda$-system. Hence the $\pi$-$\lambda$ theorem gives that $\mathcal{E}$ is the whole of $\mathcal{G} \otimes \Sigma$.
\end{proof}

\begin{defi}
We will say that \emph{$\mathcal{G}$ is conditionally trivial under $\rho$} if $\mathcal{G}$ is $\rho^\mathcal{G}(x)$-trivial for $\rho$-almost all $x \in X$.
\end{defi}

\begin{rmk}
Note that this is weaker than saying that $\rho^\mathcal{G}(x)|_{\mathcal{G}}=\delta_x|_{\mathcal{G}}$ for $\rho$-almost all $x \in X$. (For more on this latter condition, see [2], noting also the corrections in [3].) Conditional triviality has been considered in [4] and [5], and is also the key concept in the measure-theoretic versions of the ``ergodic decomposition theorem'' (see e.g.~Theorem~10.1 of [6]). The following example (essentially taken from [4]) demonstrates that conditional triviality does not automatically hold: Let $X=[0,1] \times \{0,1\}$ (with $\Sigma=\mathcal{B}([0,1]) \otimes 2^{\{0,1\}}$), let $\rho$ be a probability measure on $X$ whose projection onto $[0,1]$ is atomless, and let $\mathcal{G} \subset \Sigma$ be the smallest $\sigma$-algebra containing both $\mathcal{B}([0,1]) \otimes \{\emptyset,\{0,1\}\}$ and all the $\Sigma$-measurable $\rho$-null sets. Then given any non-Dirac probability measure $m$ on the binary set $\{0,1\}$, the function $\rho^\mathcal{G}:(x,i) \mapsto \delta_x \otimes m$ is a rcd of $\rho$ given $\mathcal{G}$. Clearly, for any $x \in [0,1]$, $\{(x,0)\} \in \mathcal{G}$; and yet, for all $(x,i) \in X$,
\[ \rho^\mathcal{G}(x,i)(\{(x,0)\}) \ = \ m(0) \, \in \, (0,1). \]
\end{rmk}

\begin{defi}
A function $\rho^\mathcal{G}_2:X \times X \to \mathcal{P}$ will be called an \emph{iterated rcd of $\rho$ given $\mathcal{G}$} if for $\rho$-almost every $x \in X$, the map $y \mapsto \rho^\mathcal{G}_2(x,y)$ is a rcd of $\rho^\mathcal{G}(x)$ given $\mathcal{G}$.
\end{defi}

\noindent (Note that this definition is independent of the version of $\rho^\mathcal{G}$ that we are working with.)

\begin{rmk}
If $(X,\Sigma)$ is standard then the axiom of choice yields that there must exist an iterated rcd of $\rho$ given $\mathcal{G}$ (although we do not automatically know about the measurability properties thereof).
\end{rmk}

\begin{thm}
$\mathcal{G}$ is conditionally trivial under $\rho$ if and only if there exists an iterated rcd of $\rho$ given $\mathcal{G}$ which is $(\mathcal{G} \otimes \mathcal{G},\mathfrak{K})$-measurable.
\end{thm}

\begin{proof}
If $\mathcal{G}$ is conditionally trivial under $\rho$ then (as in Remark~2) the map $(x,y) \mapsto \rho^\mathcal{G}(x)$ is an iterated rcd of $\rho$ given $\mathcal{G}$, and it is obviously $(\mathcal{G} \otimes \mathcal{G},\mathfrak{K})$-measurable. Conversely, suppose we have an iterated rcd $\rho^\mathcal{G}_2:X \times X \to \mathcal{P}$ of $\rho$ given $\mathcal{G}$ which is $(\mathcal{G} \otimes \mathcal{G},\mathfrak{K})$-measurable. For each $A \in \Sigma$ the set
\[ E_A \ := \ \{ (x,y) \in X \times X \, : \, \rho^\mathcal{G}_2(x,y)(A) = \rho^\mathcal{G}_2(x,x)(A) \} \]
\noindent is $(\mathcal{G} \otimes \mathcal{G})$-measurable. So, letting $\mathcal{A} \subset \Sigma$ be a countable $\pi$-system generating $\Sigma$, the set
\[ E \ := \ \{ (x,y) \in X \times X \, : \, \rho^\mathcal{G}_2(x,y) = \rho^\mathcal{G}_2(x,x) \} \ = \ \bigcap_{A \in \mathcal{A}} E_A \]
\noindent is $(\mathcal{G} \otimes \mathcal{G})$-measurable. Obviously $E$ includes the whole of the diagonal in $X \times X$, and therefore (by Lemma~3)
\[ \int_X \int_X \mathbbm{1}_E(x,y) \, \rho^\mathcal{G}(x)(dy) \, \rho(dx) \ = \ 1. \]
\noindent It follows that for $\rho$-almost all $x \in X$,
\[ \rho^\mathcal{G}(x)(y \in X : \rho^\mathcal{G}_2(x,y)=\rho^\mathcal{G}_2(x,x)) \ = \ 1. \]
\noindent So, for $\rho$-almost every $x \in X$, the map $y \mapsto \rho^\mathcal{G}_2(x,y)$ is $\rho^\mathcal{G}(x)$-almost constant. Thus (by Remark~2 applied to $\rho^\mathcal{G}(x)$) $\mathcal{G}$ is $\rho^\mathcal{G}(x)$-trivial for $\rho$-almost all $x \in X$.
\end{proof}

\begin{rmk}
In view of the above theorem, the fact that conditional triviality does not always hold serves to demonstrate the ``non-measurability'' (or, in heuristic terms, the ``non-constructivity'') of the operation of ``conditioning with respect to a $\sigma$-algebra''. (Indeed, conditional probabilities may be obtained as either Radon-Nikodym derivatives or orthogonal projections in $\mathcal{L}^2$, both of which are ``non-constructively'' proven to exist.) This point can be made slightly more precise as follows: Suppose $\mathcal{G}$ has the property that there exists a $(\mathfrak{K} \otimes \mathcal{G},\mathfrak{K})$-measurable function $C^\mathcal{G}:\mathcal{P} \times X \to \mathcal{P}$ such that for all $\mu \in \mathcal{P}$ the map $x \mapsto C^\mathcal{G}(\mu,x)$ is a rcd of $\mu$ given $\mathcal{G}$; then the map $(x,y) \mapsto C^\mathcal{G}(C^\mathcal{G}(\rho,x),y)$ is clearly a $(\mathcal{G} \otimes \mathcal{G},\mathfrak{K})$-measurable iterated rcd of $\rho$ given $\mathcal{G}$, and therefore $\mathcal{G}$ is conditionally trivial under $\rho$.
\end{rmk}

\subsection*{Acknowledgement}

The author would like to thank Prof Patrizia Berti and Prof Pietro Rigo for their helpful comments and kind encouragement.

\subsection*{References}

[1] Billingsley, P., \emph{Probability and Measure} (3rd ed.), Wiley, New York (1995).
\\ \\
\textrm{[2]} Seidenfeld, T., Schervish, M. J., Kadane, J. B., Improper Regular Conditional Distributions, \emph{Ann. Prob.} \textbf{29}(4), pp1612--1624 (2001).
\\ \\
\textrm{[3]} Seidenfeld, T., Schervish, M. J., Kadane, J. B., Correction: Improper Regular Conditional Distributions, \emph{Ann. Prob.} \textbf{34}(1), pp423--426 (2006).
\\ \\
\textrm{[4]} Berti, P., Rigo, P., 0-1 Laws for Regular Conditional Distributions, \emph{Ann. Prob.} \textbf{35}(2), pp649--662 (2007).
\\ \\
\textrm{[5]} Berti, P., Rigo, P., A Conditional 0-1 Law for the Symmetric $\sigma$-field, \emph{J. Theor. Probab.} \textbf{21}(3), pp517--526 (2008).
\\ \\
\textrm{[6]} Gray, R. M., \emph{Probability, Random Processes, and Ergodic Properties} (2nd ed.), Springer, Dordrecht (2009).

\end{document}